\documentclass[11pt,twoside,reqno]{amsart}

\usepackage[left=2.9cm,top=3.5cm,right=2.9cm]{geometry}
\usepackage[english]{babel}
\usepackage[applemac]{inputenc}
\usepackage[T1]{fontenc}
\usepackage{microtype}
\usepackage{cite}
\usepackage{float, verbatim}
\usepackage{amsmath}
\usepackage{amssymb}
\usepackage{amsthm}
\usepackage{amsfonts}
\usepackage{graphicx}
\usepackage{mathtools}
\usepackage{enumitem}
\usepackage{color}
\usepackage{tikz}
\usepackage{xcolor}
\usepackage{pgfplots}
\usepackage{caption}
\usepackage{subcaption}
\usepackage{enumerate}

\pgfplotsset{compat=1.14}

\newcommand{\Haus}{\dim_{\mathrm{H}}}
\newcommand{\hd}{\dim_{\mathrm{H}}}
\newcommand{\pd}{\dim_{\mathrm{P}}}
\newcommand{\Lower}{\dim_{\mathrm{L}}}

\newcommand{\Assouad}{\dim_{\mathrm{A}}}
\newcommand{\LBox}{\underline{\dim}_{\mathrm{B}}}
\newcommand{\UBox}{\overline{\dim}_{\mathrm{B}}}

\newcommand{\boxd}{\dim_\mathrm{B}}
\newcommand{\uboxd}{\overline{\dim}_\mathrm{B}}
\newcommand{\lboxd}{\underline{\dim}_\mathrm{B}}

\numberwithin{equation}{section}

\newtheorem*{thm*}{Theorem}
\newtheorem*{conj*}{Conjecture}
\newtheorem*{rem*}{Remark}
\newtheorem*{lma*}{Lemma}

\newtheorem{thm}{Theorem}[section]

\newtheorem{lma}[thm]{Lemma}
\newtheorem{cor}[thm]{Corollary}
\newtheorem{defn}[thm]{Definition}

\newtheorem{prop}[thm]{Proposition}

\newtheorem{ques}[thm]{Question}

\begin{document}
	
	\title{On the Hausdorff dimension of microsets}
	
	\author[J. M. Fraser]{Jonathan M. Fraser}
	\address{Jonathan M. Fraser\\
		School of Mathematics \& Statistics\\University of St Andrews\\ St Andrews\\ KY16 9SS\\ UK }
	\curraddr{}
	\email{jmf32@st-andrews.ac.uk}
	
	\author[D. C. Howroyd]{Douglas C. Howroyd}
	\address{Douglas C. Howroyd\\
		School of Mathematics \& Statistics\\University of St Andrews\\ St Andrews\\ KY16 9SS\\ UK }
	\curraddr{}
	\email{dch8@st-andrews.ac.uk}
	
	\author[A. K\"aenm\"aki]{Antti K{\"a}enm{\"a}ki}
	\address{Antti K{\"a}enm{\"a}ki\\
		Department of Physics and Mathematics\\ University of Eastern Finland\\ P.O. Box 111\\ FI-80101 Joensuu\\ Finland }
	\curraddr{}
	\email{antti.kaenmaki@uef.fi}
	
	\author[H. Yu]{Han Yu}
	\address{Han Yu\\
		School of Mathematics \& Statistics\\University of St Andrews\\ St Andrews\\ KY16 9SS\\ UK }
	\curraddr{}
	\email{hy25@st-andrews.ac.uk}
	
	\thanks{}

	
	
	\date{}
	
	\dedicatory{}
	
	\begin{abstract}
		We investigate how the Hausdorff dimensions of microsets are related to the  dimensions of the original set. It is known that the maximal dimension of a microset is the Assouad dimension of the set. We prove that the lower dimension can analogously be obtained as the minimal dimension of a microset. In particular, the maximum and minimum exist.  We also show that for an arbitrary $\mathcal{F}_\sigma$ set $\Delta \subseteq [0,d]$ containing its infimum and supremum there is a compact set in $[0,1]^d$ for which the set of Hausdorff dimensions attained by its microsets is  exactly equal to  the set $\Delta$.  Our work is motivated by the general programme of determining what geometric information about a set can be determined at the level of tangents.
	\end{abstract}
	
	\maketitle
	\section{Introduction}
	To calculate the dimension of a set it is often important to understand its infinitesimal structure. This leads us to the notion of microsets introduced by Furstenberg \cite{Fu}. They are sets that are obtained as limits of successive magnifications of the original set. From a dynamical point of view, the collection of all microsets together with the magnification action define a dynamical system. The study of this dynamical system is known as the theory of CP-chains. For more details in this direction, see also \cite{FFS,Fu,H10,HS12,KSS15}. In this paper, we want to study the collection of \emph{all} microsets. This collection heuristically represents all possible fine structures  of a set. For general compact sets the structure of this collection is very rich; see \cite{CR}.
	
	The Assouad dimension characterises how large the densest part of a set is. It is known that the greatest Hausdorff dimension of all microsets of a set $F$ is equal to the Assouad dimension of $F$. In much the same way, the lower dimension reflects how sparse a set can be and it is natural to expect that the smallest microset of a set $F$ represents the lower dimension of $F$. This is our first result.  
	
	\begin{thm}\label{ThMain}
		For any compact set $F\subset \mathbb{R}^d$ we have
		\[
		\Lower F=\min_{E\in\mathcal{G}_F} \Haus E=\min_{E\in\mathcal{G}_F} \uboxd E.
		\]
		In particular, this minimum exists.
	\end{thm}
	
	Here $\Lower$ stands for the lower dimension, $\Haus$ for the Hausdorff dimension, $\uboxd$ for the upper box dimension, and $\mathcal{G}_F$ for the gallery of $F$; see Section \ref{prelim} for the precise definitions. Combining this result with the analogous one for the Assouad dimension, we obtain the following corollary.
	
	\begin{cor}\label{maincor}
		For any compact set $F\subset \mathbb{R}^d$, all elements in $\mathcal{G}_F$ have the same Hausdorff dimension if and only if
		\[
		\Lower F=\Assouad F.
		\]
	\end{cor}
	
	Here $\Assouad$ stands for the Assouad dimension; again see Section \ref{prelim} for the definition.
	
	We know that the Hausdorff dimension of microsets attains both the lower and Assouad dimensions of a set. The question then becomes which other numbers \emph{can} be attained by the dimensions of microsets and which numbers are \emph{guaranteed} to be attained. The next result shows that the collection of Hausdorff dimensions obtained can be rather complicated and rich.
	
	\begin{thm} \label{Thmain2}
		If $\Delta \subseteq [0,d]$ is an $\mathcal{F}_\sigma$ set which contains its infimum and supremum, then there exists a compact set $F \subseteq [0,1]^d$ such that
		\[
		\{ \hd E : E \in \mathcal{G}_F \} = \Delta.
		\]
	\end{thm}
	
	The gallery of a set is closed under the Hausdorff metric. However, the above theorem says that the set of dimensions of microsets in the gallery need not be.  We have not been able to construct compact sets $F$ for which $\{ \hd E : E \in \mathcal{G}_F \} $ is not $\mathcal{F}_\sigma$ and wonder if this is always the case.
	
	Note that the   Hausdorff, packing, upper and lower box dimensions of the original set need not appear as Hausdorff dimensions of sets in the gallery if one insists on the microsets having unbounded scaling sequence.  This is a natural assumption which guarantees that microsets genuinely reflect infinitesimal structure, that is, one genuinely zooms in to generate them.  This is in stark contrast to the Assouad and lower dimensions which we have seen always appear. See Section \ref{dimappear} for a full discussion of this observation.
	
	In the opposite direction, there exist well studied sets whose microsets observe all possible dimensions between the lower and Assouad dimensions. For instance, Bedford-McMullen carpets $F$ which do not have uniform fibers have the property that $\Lower F< \Assouad F$ and  $\{ \hd E : E \in \mathcal{G}_F \} = [\Lower F, \Assouad F]$.  This can be seen by adapting the arguments in \cite{mackay, F} which construct extremal microsets to such carpets.  The extremal microsets are of the form $\pi F \times C$ where $\pi F$ is the projection of $F$ onto the first coordinate and $C$ is a self-similar set corresponding to the minimal or maximal column.  To obtain intermediate dimensions, one may construct microsets of the form $\pi F \times C_p$ where $C_p$ is a `random Cantor set', where the minimal column is chosen with probability $(1-p)$ and the maximal column with probability $p$. Varying $p \in (0,1)$ yields microsets with all possible dimensions. We do not pursue the details.  Alternatively the construction of Chen and Rossi \cite{CR} yields a set $F\subseteq [0,1]^d$ such that $\{ \hd E : E \in \mathcal{G}_F \} = [0,d]$.  
	
	In addition to the dimension results above, we also have the following topological result which can be naturally viewed as a dual version of \cite[Theorem 2.4]{FY1}, which says that any set of full Assouad dimension has the unit cube as a microset.
	
	\begin{thm}\label{Sing}
		For any compact set $F \subset \mathbb{R}^d$, there is a singleton in $\mathcal{G}_F$ if and only if $\Lower F = 0$.
	\end{thm}

	\section{Preliminaries}\label{prelim}
	\subsection{Dimensions}
	Let $N_r(F)$ be the smallest number of cubes of side length $r>0$ needed to cover the compact set $F \subset \mathbb{R}^d$. The \emph{upper} and \emph{lower box dimensions} of $F$ are
	\[
	\uboxd F=\limsup_{r\to 0} \frac{\log N_r(F)}{-\log r}
	\]
	and
	\[
	\lboxd F=\liminf_{r\to 0}\frac{\log N_r(F)}{-\log r},
	\]
	respectively. When these two values coincide we simply talk about the \emph{box dimension} of $F$, denoted by $\boxd F$.
	
	Let $F\subset \mathbb{R}^d$ be a compact set and $s$ a non-negative real. For all $\delta>0$ we define
	\[
	\mathcal{H}^s_\delta(F)=\inf\left\{\sum_{i=1}^{\infty}\mathrm{diam} (U_i)^s: F \subset \bigcup_i U_i \text{ and } \mathrm{diam}(U_i)<\delta\right\}.
	\]
	The $s$-dimensional Hausdorff measure of $F$ is
	\[
	\mathcal{H}^s(F)=\lim_{\delta\to 0} \mathcal{H}^s_{\delta}(F)
	\]
	and the \emph{Hausdorff dimension} of $F$ is
	\[
	\Haus F=\inf\{s\geq 0:\mathcal{H}^s(F)=0\}=\sup\{s\geq 0: \mathcal{H}^s(F)=\infty          \}.
	\]
	For a more thorough treatment of the box and Hausdorff dimensions, see \cite[Chapters 2 and 3]{Fa} and \cite[Chapters 4 and 5]{Ma1}. 
	
	Finally we define the \emph{Assouad} and \emph{lower dimensions} of $F$ by
	\begin{align*}
	\Assouad F = \inf \Bigg\{ s \ge 0 \, \, \colon \, (\exists \, C >0)\, (\forall & R>0)\,  (\forall r \in (0,R))\, (\forall x \in F) \\ 
	&N_r(B(x,R) \cap F) \le C \left( \frac{R}{r}\right)^s \Bigg\}
	\end{align*}
	and
	\begin{align*}
	\Lower F = \sup \Bigg\{ s \ge 0 \, \, \colon \, (\exists \, C >0)\, (\forall &\,  0<R<1)\,  (\forall r \in (0,R))\, (\forall x \in F) \\ 
	&N_r(B(x,R) \cap F) \geq C \left( \frac{R}{r}\right)^s \Bigg\},
	\end{align*}
	where $B(x,r)$ is the closed ball of centre $x$ and radius $r$. For basic properties of these dimensions, see \cite{F}.
	
	
	The main property we will use is that,
	\[
	\Lower F\leq \Haus F\leq \LBox F\leq \UBox F\leq \Assouad F
	\]
	for all compact $F\subset\mathbb{R}^d$.

	\subsection{Microsets and galleries}\label{MG}
	
	We now introduce the notion of microsets and galleries following \cite{Fu}. We start by defining the Hausdorff distance between two compact sets $A,B \subset \mathbb{R}^d$ by
	\[
	d_{\mathcal{H}}(A,B)=\inf\{\delta>0: A\subset B_{\delta} \text{ and } B\subset A_{\delta}\},
	\]
	where $E_\delta$ is the closed $\delta$-neighbourhood of a compact set $E$. 
	
	Let $X=[0,1]^d$ for some $d\in \mathbb{N}$. Then $(\mathcal{K}(X),d_\mathcal{H})$, the space of compact subsets of $X$, is a compact metric space.
	
	\begin{defn}
		We call $D\in \mathcal{K}(X)$ a \emph{miniset} of $F\in\mathcal{K}(\mathbb{R}^d)$ if $D = \left(\lambda F + t\right)\cap X$ for some scaling coefficient $\lambda \ge 1$ and a translation vector $t\in \mathbb{R}^d$. Maps of the form $T(x) =  \lambda x + t$ will be called homotheties.
		A set $E \in \mathcal{K}(X)$ is called a \emph{microset} if it is a limit of a sequence $(D_n)_{n \in \mathbb{N}}$ of minisets under the Hausdorff metric. The sequence $(\lambda_n)_{n \in \mathbb{N}}$, where each $\lambda_n$ is a scaling coefficient of the miniset $D_n$, is called the \emph{scaling sequence} of the microset $E$.
	\end{defn}
	
	If a set is regular enough, for instance a self-similar set of positive dimension, then one could expect all microsets to be of the same dimension without appealing to Corollary \ref{maincor}. However, it is easy to find microsets which are just singletons. Simply consider the middle third Cantor set $\mathcal{C}$, then $(4\mathcal{C} - 4/3) \cap [0,1]= \{0\}$ is a miniset and hence a microset. This example can be easily modified such that none of the defining minisets contain singletons but the microsets do. Thus it is natural to discard all microsets which only contain points on the boundary of $X$. This will be reflected in the next definition which strays slightly from the formulation in \cite{Fu}.
	
	\begin{defn} \label{def:gallery}
		Let $F$ be a compact subset of $\mathbb{R}^d$. We consider only microsets which intersect the interior of $X$. Then the collection of all such microsets of $F$ is called the gallery of $F$, denoted by $\mathcal{G}_F$.
	\end{defn}
	
	Due to \cite{MT} we know that for compact subsets $F\subset \mathbb{R}^d$,
	\begin{equation} \label{eq:assouad-upper-micro}
	\Assouad F \ge \sup_{E\in\mathcal{G}_F} \Assouad E.
	\end{equation}
	The lower dimension case was considered in \cite[Proposition 7.7]{F} where the following proposition was obtained under some extra assumption but with the infimum taken over lower dimensions of microsets. We give a short proof to show that the extra assumption is not needed when one takes the infimum over Hausdorff dimensions.
	
	\begin{prop} \label{prop:lower-lower-bound}
		If $F \subseteq [0,1]^d$ is a compact set, then
		\[
		\Lower F \le \inf_{E\in\mathcal{G}_F} \Haus E.
		\]     
	\end{prop}
	
	\begin{proof}
		Let $F\subseteq [0,1]^d$ be compact. We may assume $\Lower F> 0$ since otherwise there is nothing to prove. If $0<s< \Lower F$, then for any sequence of homotheties $T_k \colon \mathbb{R}^d \rightarrow \mathbb{R}^d$ there exists a constant $C >0$ such that for any $k\in \mathbb{N}$, $y\in T_k(F)$ and $0<r<R< 1$ we have
		\[
		N_r(B(y,R)\cap T_k(F) \cap [0,1]^d) \ge C\left( \frac{R}{r}\right)^s.
		\]
		If this was not true, then the lower dimension of $F$ would be strictly less than $s$, a contradiction. Note that $T_k(F) \cap [0,1]^d$ is a miniset of $F$.
		
		Let $E\in \mathcal{G}_F$ and recall from Definition \ref{def:gallery} that then $E \cap (0,1)^d \neq \emptyset$. Now choose $T_k$ such that $E$ is the limit of $T_k(F) \cap [0,1]^d$. Note that if $E$ has an isolated point on the boundary of $[0,1]^d$, then $\Lower E = 0 < \Lower F$. To show that $\Lower F \le \Lower E \cap (0,1)^d$ let us fix $0<r<R<1$ and $x \in E \cap (0,1)^d$. Choose $k\in \mathbb{N}$ so that $d_{\mathcal{H}}(T_k(F) \cap [0,1]^d, E) \le r/2$. Then there is $y \in T_k(F) \cap [0,1]^d$ such that $B(y,R/2) \subset B(x,R)$ and for every $r$-cover of $B(x,R) \cap E \cap (0,1)^d$ there is a $2r$-cover of $B(y,R/2) \cap T_k(F) \cap [0,1]^d$ having at most the same cardinality. Thus
		\begin{align*}
		  N_r(B(x,R)\cap E \cap (0,1)^d) &\ge N_{2r}(B(y,R/2)\cap T_k(F) \cap [0,1]^d) \\ &\ge C4^{-s} \left(\frac{R}{r} \right)^s
		\end{align*}
		and so $\Lower E \cap (0,1)^d \ge s$, yielding
		\[
		\Lower F \le \Lower E \cap (0,1)^d \le \Haus E
		\]
		as desired.
	\end{proof}

	We are interested in whether the above inequalities are actually equalities and if the supremum and infimum can be attained. For the Assouad dimension, we have the following result.
	
	\begin{thm} \label{MicroAssouad}
		If $F \subset \mathbb{R}^d$ is a compact set, then
		\[
		\Assouad F=\max_{E\in\mathcal{G}_F} \Haus E=\max_{E\in\mathcal{G}_F} \Assouad E.
		\]
	\end{thm}
	
	\begin{proof}
		Recalling \eqref{eq:assouad-upper-micro}, the statement follows from \cite[Theorem 5.1]{Fu} and \cite[Proposition 3.13]{KR}, or, alternatively, directly from \cite[Proposition 5.7]{KOR}.
	\end{proof}
	
	Thus, together with Theorem \ref{ThMain}, we obtain the following equivalent definitions of the Assouad and lower dimensions for compact subsets of Euclidean spaces
	\[
	\Assouad F=\max_{E\in\mathcal{G}_F} \Haus E
	\]
	and
	\[
	\Lower F=\min_{E\in\mathcal{G}_F} \Haus E.
	\]
	We remark that in the literature weak tangents are often used in place of microsets. They differ from microsets by allowing rotations in the magnifications, and sometimes they are not restricted to the unit cube. 

	\section{Global and local size of trees}\label{Tree}
	Before proving Theorem \ref{ThMain} we need some combinatorial results on the structure of trees. Here we only talk about binary trees (each vertex has at most two children) but all definitions and results can be easily generalized to any $k$-ary trees with $k\geq 3$. Notation introduced in this section will only be used in this section and the next one.
	
	We adopt standard graph theoretic notation and use $V(T)$ and $E(T)$ for vertices and edges of $T$. We consider trees as directed graphs, with edges going down to descendants. We define the degree of a vertex to be the sum of indegrees and outdegrees of a vertex. A leaf of $T$ is an element in $V(T)$ whose degree is $1$, except when $T$ consists of just the root vertex (the vertex which is not a descendant of any other vertices) then the only leaf is the root of degree 0. We denote the set of leaves of $T$ as $L(T)$, so in particular $L(T)\subset V(T)$. We say that two trees are equal if they are isomorphic in terms of directed graph.
	
	Given a binary tree $T$, the height $h(T)$ is the length of the longest path starting at the root vertex. When $h(T)<\infty$ we say that $T$ is finite. For a vertex $a\in V(T)$, $h(a)$ is the length of the unique path from the root to $a$. We will often use the term `level $n$' to mean all vertices of height $n$. Given $a\in V(T)$, we use $T(a,n)$ to denote the largest subtree of $T$ with root $a$ of height at most $n$, formally the vertex set of $T(a,n)$ is defined to be as follows
	\[
	V(T(a,n))=\{b\in V(T): \text{there is a path in $T$ from $a$ to $b$ of length at most $n$}\}.
	\]
	The edge set of $T(a,n)$ is defined to be as follows
	\[
	E(T(a,n))=\{(b_1,b_2)\in E(T):b_1\in V(T(a,n)), b_2\in V(T(a,n)) \}.
	\]
	In other words, $T(a,n)$ is the spanned subgraph of $T$ with vertices $V(T(a,n))$. 
	
	For any binary tree $T$ we use $\#T$ to denote the number of leaves and $\#_nT$ for the number of vertices of height $n$. A binary tree $T$ is tidy if $h(a)=h(T)$ for all leaves $a\in V(T)$. For example, a full tree (all non-leaf vertices have two children) is tidy but not vice versa. If $T$ is tidy, then for any $a\in V(T)$ and any integer $n$ such that $h(a)+n\leq h(T)$, it is clear that $T(a,n)$ is tidy. Note that if $h(a)+n>h(T),$ then $T(a,n)$ is not defined as a subtree of $T$.
	
	\begin{defn}\label{loc}
		Let $T$ be a tidy binary tree, $s>0$ and $m\in \mathbb{N}$. We call $T$ locally $(s,m)$-large (or small) if for all $a\in V(T)$ with $h(a)+m\leq h(T)$, there exists $1\leq n\leq m$ such that 
		\[
		\#T(a,n)\geq 2^{sn}
		\quad\left( \text{or } \#T(a,n)\leq 2^{sn}  \right).
		\]
	\end{defn}

Note that when $T$ is infinite then this must simply hold for all $a\in V(T)$.  Roughly speaking, a  locally $(s,m)$-large (or small) tree is one such that below every vertex there is a tree with height less than $m$ which is big (or small) quantified by $s$.  
	
	\begin{defn}\label{glo}
		Let $T$ be a tidy binary tree, $s>0$ and $C>0$. We call $T$ globally $(s,C)$-large (or small) if for all $n\in [1,h(T)]$ 
		\[
		\#_n T\geq C2^{sn}
		\quad\left( \text{or } \#_n T\leq C2^{sn}\right).
		\]
	\end{defn}
	Again note that if $T$ is infinite then this must hold for all $n\in [1,\infty)$. Roughly speaking, a  globally  $(s,m)$-large (or small) tree is one which is large (or small) at every level quantified by $s$. We state and prove our regularity lemma in terms of largeness. Note that it is also possible to obtain an analogous lemma with largeness being replaced by smallness. The proof is similar and we omit the details.
	
	\begin{lma}\label{reg}
		Let $T$ be a tidy locally $(s,m)$-large tree with height larger than $m$, then it is globally $(s,2^{-sm})$-large as well.
	\end{lma}
	\begin{proof}
		As $T$ is locally large in the above sense we can use the following algorithm to find large subtrees.
		
		\textbf{Step 1}: let $T_0$ be the tree whose vertex set contains only the root of $T$.
		
		\textbf{Step 2}: If $T_k$ is defined for an integer $k$, then $T_k$ has leaves. Take a leaf $a$ of $T_k$, then there is an integer $1\leq n\leq m$ such that
		\[
		\#T(a,n)\geq 2^{sn}.
		\]
		Then we join $T(a,n)$ to $T_k$ at $a$ and call the tree obtained $T_{k+1}$.
		
		We can repeat Step 2 for all leaves of $T_k$. Notice that the above algorithm is not deterministic as there are multiple choices of leaves in Step 2. However, the algorithm could easily be made deterministic by picking leaves ``from left to right''.
		
		Let $N$ be an integer larger than $m$ and not greater than the height of $T$. Let $T_N$ be the subtree of $T$ obtained by repeating Step 2 as many times as possible while keeping leaf height greater than $N$. Note that $T^N$ is maximal in the sense that we cannot enlarge $T^N$ by applying Step 2. Then it is clear that all the leaves of $T^N$ have height at least $N-m$ for otherwise the local largeness can help us enlarge $T^N$.

		We define the $s$-weight of $a\in V(T)$ to be $W(a) = 2^{-sh(a)}$ for $s\in [0,1]$. Note that the root always has $s$-weight 1. Consider the total $s$-weight of the leaves of $T^N$:
		\[
		W(T^N)=\sum_{a \in L(T^N)} 2^{-sh(a)}.
		\]
		We can group the leaves of $T^N$ according to their ancestors. Namely, for each leaf $a$ of $T^N$, there is a unique $b\in V(T^N)$ such that $a$ is a leaf of $T(b,n)$ for some $n\in [1,m]$ and $T(b,n)$ is a tree joined to the main tree in Step 2 of our algorithm. We shall denote $b=b(a)$ to imply this dependence. 
		
		We now compute the total $s$-weight of the leaves of $T^N$. First, we find the set $L$ of leaves with maximum height. This is possible because there are only finitely many leaves in $L(T^N)$. Then if $a\in L$ and $b=b(a)$ then $T(b,h(a)-h(b))$ is contained in $T^N$ and because $T$ is tidy, $L(T(b,h(a)-h(b)))\subset L$. As there are only finitely many leaves in $L$ we can find a finite collection $B$ of vertices $b\in V(T^N)$ and integers $\{n_b\}_{b\in B}$ such that $L$ is the disjoint union of sets $L(T(b,n_b)), b\in B$. Therefore we see that
		\[
		\sum_{a\in L} 2^{-sh(a)}=\sum_{b\in B}\sum_{a\in L(T(b,n_b))} 2^{-sh(a)}.
		\]
		As $T$ is locally $(s,m)$-large, we see that for each $b\in B$
		\[
		\sum_{a\in L(T(b,n_b))} 2^{-s(h(a))}\geq 2^{-s(h(b)+n_b)}2^{s(n_b)}=2^{-sh(b)}.
		\]
		This implies that
		\[
		\sum_{a\in L} 2^{-sh(a)}\geq \sum_{b\in B} 2^{-sh(b)}.
		\]
		Now we construct a subtree $T^N_1$ of $T^N$ replacing the subtrees $T(b,n_b), b\in B$ with single vertices $b\in B$. Then we have seen from above that
		\[
		\sum_{a\in L(T^N_1)} 2^{-sh(a)}\leq \sum_{a'\in L(T^N)} 2^{-sh(a')},
		\]
		because each $a\in L(T^N_1)$ is either in $L(T^N)$ or else it is in $B$.
		
		We can then perform the above procedure on the tree $T^N_1$ instead of $T^N$ and we obtain a subtree $T^N_2$ whose leaves have weight no greater than that of $T^N_1$. Moreover, the height of $T^N_2$ is strictly smaller than the height of $T^N_1$. This means that after performing the above procedure at most finitely many times we arrive at the tree with only one vertex, the root. This implies that
		\[
		W(T^N)\geq 1.
		\]
		
		As we just observed, the leaves of $T^N$ have height at least $N-m$ so their weights are at most $2^{-s(N-m)}$ so the number of the leaves is at least
		\[
		2^{s(N-m)}.
		\] 
		However, observe that $\#_N T\geq \# T^N$ because $T$ is tidy. Therefore we see that
		\[
		\#_N T \geq 2^{s(N-m)}.
		\]
		As $N$ is arbitrarily chosen, this is what we want.
	\end{proof}

	\section{Lower  dimension and microsets} \label{Microset}
	
	Returning to the Euclidean space, we now prove Theorem \ref{ThMain}. We shall show that there exists a microset $E \in \mathcal{G}_F$ such that
	\[
	\uboxd E \le \Lower F
	\]
	whenever $F$ is a compact subset of $[0,1]$. The result easily generalizes to higher dimensions and Theorem \ref{ThMain} then follows from Proposition \ref{prop:lower-lower-bound}. The main idea behind the proof is to represent subsets of the unit interval as dyadic trees and then use the previous regularity lemma to determine the covering number of a microset.
	
	For $n\in \mathbb{N}$ and $i\in \left\{0,1,\ldots, 2^n-1 \right\}$ we define the $i^\textrm{th}$ dyadic interval of height $n$ to be $D_n(i) = \left[\frac{i}{2^n}, \frac{i+1}{2^n}\right]$. This interval is then associated with the $i^\textrm{th}$ vertex of level $n$ in the full binary tree. We can then associate a subtree $T(F)$ of the full binary tree to a compact set $F\subseteq [0,1]$  by removing the $j^\textrm{th}$ vertex of level $k$ (as well as all of its descendants) if $D_k(j) \cap F = \emptyset$. Note that if $D_k(j) \cap F = \emptyset$ for some $k$ and $j$ then any smaller dyadic interval inside $D_k(j)$ must also not intersect $F$. So $T(F)$ is indeed a subtree of the full dyadic tree. Later on we will use subtrees to find microsets. In order to satisfy the condition that microsets intersect the interior of the reference set $X$, we need to modify $T(F)$ as follows. If there is some $n\in \mathbb{N}$ and $i\in \{1,\ldots,2^n-1\}$ such that $\frac{i}{2^n}\in F$ then we need to check whether $\left(\frac{i-1}{2^n}, \frac{i}{2^n}\right)\cap F = \emptyset$ or $\left(\frac{i}{2^n}, \frac{i+1}{2^n}\right)\cap F = \emptyset$. If both intersections are empty then without loss of generality we remove the vertex associated to $D_n(i)$ and keep the vertex associated to $D_n(i-1)$.  If both are non-empty then we keep both vertices.  Finally, and most importantly, if only one of the two intersections is empty then we remove the vertex associated with the dyadic interval forming the empty intersection. It is straightforward to check that $T(F)$ is a tidy, infinite, binary tree. 
	
	The inverse of the above procedure can be described as follows. Given a tidy, infinite binary tree $T$, we can associate a compact set $S(T)$ in a natural way by identifying vertices of $T$ as dyadic intervals. Given an infinite path $l$ in $T$ from the root, we can find uniquely a point $x$ contained in the intersection of all dyadic intervals corresponds to the vertices along $l$. In this way we see that for each compact subset $F$ of $[0,1]$, $S(T(F))=F$. We note that it is not true that $T(S(Tr))=Tr$ holds for all tidy, infinite binary tree $Tr.$

	Since we wish to compare microsets and trees, we must also have a suitable notion of convergence of trees. Let $T_i$ be a sequence of binary trees with roots denoted by $a_i$. We say that $T_i$ converges if there exists a sequence of tidy binary trees $\{K_n\}_{n\in\mathbb{N}}$ with the same root $a$ and height $n$ such that for all $n\in\mathbb{N}$ there exists an $I\in \mathbb{N}$ such that
	\[
	T_i(a_i,n) \text{ are equal to $K_n$ for all $i\ge I$}.
	\]
	The limit $\lim_{i\to\infty} T_i=T$ is defined to be the binary tree with root $a$ and $T(a,n)=K_n$ for all $n$. Notice that if the above holds then it is necessary that $K_{n_1}$ is a subtree of $K_{n_2}$ for integers $n_1\leq n_2$.
	
	For any sequence of binary trees with unbounded heights, there exists a convergent subsequence. To see this we note that if a tree has some number of vertices at some level then there are only finitely many configurations for the vertices on the next level. Therefore for any sequence of trees of unbounded height, there will always be at least one configuration of the first $n$ levels that repeats infinitely often for all $n\in \mathbb{N}$. 
	
	Let $T=T(F)$, then for any $a\in V(T)$ and integer $n$, the subtree $T(a,n)$ corresponds to a finite approximation of a miniset $E$ of $F$ by blowing a dyadic interval of length $2^{-h(a)}$ up to the unit interval. Given a convergent sequence $T(a_i,n_i)$ with $h(a_i)\to\infty$, we can find a convergent sequence of minisets $E_i=S(T(a_i,\infty))$ such that the binary tree associated with the limit $E_{\infty}$ is precisely $\lim_{i\to\infty} T(a_i,n_i)$. To see that $E_i$ indeed converges to $S(\lim_{i\to\infty} T(a_i,n_i))$ we need only to see that $n_i\to\infty$ and the Hausdorff metric between $E_i$ and $S(\lim_{i\to\infty} T(a_i,n_i))$ is bounded from above by $2^{-n_i}$. 
	
	Due to the construction of our tree such microsets will only intersect the boundary of $X$ if there is a genuine isolated point, in which case there exists a number of actual microsets in the gallery containing the isolated point. Thus, without loss of generality, we may assume that the microsets obtained do satisfy our extra condition (that they intersect the interior of $X$).
	
	\begin{lma}
		If $F \subseteq [0,1]$ is a compact set and $\varepsilon > 0$, then there is a microset $E \in \mathcal{G}_F$ such that $\UBox E \le \Lower F+\varepsilon$.
	\end{lma}
	
	\begin{proof}
		We associate a binary tree $T(F)$ to $F$. Such a tree $T(F)$ is tidy by construction. Write $s = \Lower F$ and observe that, for any $\varepsilon>0$, we can find tidy subtrees $T_i=T(a_i,n_i)$ of $T(F)$ with height $n_i$ and $\#_{n_i} T_i\leq 2^{n_i(s+\varepsilon)}$. Moreover, we can assume that $n_i\to\infty$ as $i\to\infty$. We can also assume that $h(a_i)\to\infty$. Indeed, if it is not possible to find such a sequence of $a_i$, then there is an integer $N_0$ such that $\#_{N} T(a, N)\geq 2^{(s+\varepsilon)N}$ whenever $N\geq N_0$ and $h(a)\geq N_0$. This implies that $\Lower F\geq s+\varepsilon$ which is not possible.
		
		Let us show that, for any integer $m\geq 1$, we can find $T'_m=T(a'_m,n'_m)$ such that $a'_m\in\bigcup_{i\in\mathbb{N}} V(T_i)$ and
		\begin{equation} \label{eq:subtree}
		\#_n T'_m\leq 2^{(s+2\varepsilon)n}    
		\end{equation}
		for all $n\in [1,m]$. If we cannot find such a collection of subtrees, then all the trees $T_i$ are locally $(s+2\varepsilon,m)$-large for an integer $m$ which does not depends on $i$. Of course by dropping finitely many $T_i$ we may assume that $h(T_i)\geq m$ for all $i$. Then take a $T_i$ with large height (compared to $m$), then by Lemma \ref{reg} we see that $\#_{n_i}(T_i)\geq 2^{(s+2\varepsilon)(n_i-m)}$. So we see that for all $i$, with large enough $n_i$,
		\[
		2^{(s+2\varepsilon)(n_i-m)}\leq 2^{(s+\varepsilon)n_i}.
		\]
		Therefore we see that
		\[
		n_i\leq (s+2\varepsilon)\frac{m}{\varepsilon}.
		\]
		This is a contradiction as $n_i$ can be arbitrarily large, and hence such subtrees $T'_m$ exist.
		
		Let $T'_m$ be a sequence of subtrees of $T(F)$ satisfying \eqref{eq:subtree}. By taking a subsequence of $T'_m$ if necessary we assume that the sequence converges and this corresponds to a subset $E$ of $F$ with $T(E)=\lim_{m\to\infty} T'_m$. It is clear that for large $m$, $N_{2^{-n}} (E)=\#_n T(E)=\#_n T'_m\leq 2^{(s+2\varepsilon)n}$ for all integers $n$ and therefore the upper box dimension (and so the Hausdorff dimension) of $E$ is at most $s+2\varepsilon$. Now $\sup_{m}h(a'_m)$ must be infinite because $h(a_i)\to\infty$ and there are only finitely many vertices in each $T_i=T(a_i,n_i)$. So we see that $E$ is a microset.
	\end{proof}

	\begin{thm}\label{LowerHaus}
		If $F\subseteq [0,1]$ is a compact set, then there is a microset $E \in \mathcal{G}_F$ such that $\UBox E \le \Lower F$.
	\end{thm}

	\begin{proof}
		The above lemma says we can find a microset whose upper box dimension arbitrarily approximates the lower dimension of the original set. To obtain the equality desired we will use a Cantor diagonal argument to find a sequence of minisets which converge to a set satisfying the equality. From the previous lemma we know that, given $\varepsilon > 0$, there exists a subsequence of $T'_{m,\varepsilon}$ which converges to $T_{\varepsilon}.$ Let  $E_\varepsilon=S(T_\varepsilon).$ We have $\UBox E_\varepsilon \le s + \varepsilon$, where $s=\Lower F$. Recalling \eqref{eq:subtree}, we actually have the following stronger inequality, that for all $n\in [1,m]$,
		\[
		\#_n T'_{m,\varepsilon}\leq 2^{n(s+\varepsilon)}.
		\]
		
		We construct an algorithm which will give us the desired sequence.
		\\ \\
		\textbf{Step 1}: Let $j=1$ and $n_1 = 1$.
		\\ \\
		\textbf{Step 2}: Consider the subsequence of trees $T'_{m,2^{-j}}$ which converges to $T(E_{2^{-j}})$. Let $T_{n_j,j} = T'_{n_j+k,2^{-j}}$ where $k$ is the smallest integer (including zero) such that $T'_{n_j+k,2^{-j}}$ is in the convergent subsequence.
		\\ \\
		\textbf{Step 3}: Set $n_{j+1} = n_j + k + 1$ and $j=j+1$, then repeat the previous step.
		\\ \\
		We thus obtain a sequence of strictly increasing integers $\left\{n_j \right\}$ and a sequence of trees $\left\{T_{n_j,j} \right\}$. There is therefore a subsequence which converges to the tree $T(E)$ and $E$ is such that 
		\[
		N_{2^{-n}}(E) = \#_n T(E) \le 2^{(s+2^{-j})n}
		\]
		for all $j$ and $n$. Hence $\UBox E \leq s$ as required.
	\end{proof}

\section{Obtainable Hausdorff dimensions in a gallery}
In this section we prove Theorem \ref{Thmain2}. Let $\Delta \subseteq [0,d]$ be an arbitrary $\mathcal{F}_\sigma$ set which contains its infimum and supremum which we denote by $\inf \Delta$ and $\sup \Delta$, respectively.  Further assume that $\Delta\subseteq [0,d]$ is infinite, otherwise the proof is  much simpler and we leave the details to the reader.

We use self-similar sets with particular dimensions as building blocks for $F$.  We assume a working knowledge of self-similar sets and refer the unfamiliar reader to \cite[Chapter 9]{Fa}.  In what follows, we assume that the convex hull of each self-similar set we consider is $[0,1]^d$.  Let $Q_\infty \subseteq [0,1]^d$ be a self-similar set generated by equicontractive homotheties satisfying the open set condition which has Hausdorff dimension $\sup \Delta$. Equicontractivity means that all the maps have the same contraction ratio.

\begin{lma} \label{constructionlemma}
 For each $n \in \mathbb{N}$, there is a collection $\{ K(s,n) : s \in [0,\sup \Delta] \}$ of self-similar sets in $[0,1]^d$ satisfying the open set condition such that $\hd K(s,n)=s$ and
\begin{equation} \label{1-over-n}
 d_\mathcal{H}(K(s,n), Q_\infty) \leq \sqrt{d}/n
\end{equation}
and, for a fixed $n$,
\begin{equation} \label{ensureclosed}
d_\mathcal{H}(K(s,n), K(t,n) ) \to 0
\end{equation}
as $|s-t| \to 0$.
\end{lma}

\begin{proof}
  To see that such a collection of sets exists, fix $n\in \mathbb{N}$ and assume $Q_\infty$ is generated by the iterated function system (IFS) $\left\{f_{i}\right\}_{i=1}^a$ where each $f_i$ is a contracting homothety with  common contraction ratio $c$. Define $k$ to be the smallest integer such that $c^k  \le \frac{1}{n}$. Since $Q_\infty$ has the maximal dimension $\sup \Delta$, for any $s \in (0,\sup \Delta)$, the set $K(s,n)$ can be defined to be the self-similar set satisfying $\Haus K(s,n) = s$ and generated by $a^k$ homotheties $\{g_{n,j}^s\}_{j=1}^{a^k}$ of contraction ratio $a^{-k/s} < a^{-k/  \sup \Delta }= c^k$ (where the equality follows from \cite[Theorem 9.3]{Fa}) such that for any $j\in\{1,\ldots,a^k\}$, the image of $[0,1]^d$ under $g_{n,j}^s$ lies in the corner of the image of $[0,1]^d$ under $f_{i_1} \circ f_{i_2} \circ \cdots \circ f_{i_k}$ for some $i_1,i_2,\ldots,i_k \in \{1,2,\ldots,a \}$.   To do this in a canonical way we assume that $f_{i_1} \circ f_{i_2} \circ \cdots \circ f_{i_k}(0) =g_{n,j}^s(0)$. The set $K(0,n)$ can be defined similarly where we allow contraction ratio $0$ and the set $K(\sup \Delta, n)$ is simply defined to be $Q_\infty$.   This guarantees that
\[
d_\mathcal{H}(K(s,n), Q_\infty) \leq \sqrt{d} c^k \leq \sqrt{d}/n
\]
which shows \eqref{1-over-n}.  Also, since the images $g_{n,j}^s([0,1]^d)$ are placed in the same corner for each $s$, we also get, for fixed $n$ (and therefore fixed $k$) and for $s>t$,
\[
d_\mathcal{H}(K(s,n), K(t,n) ) \leq \sqrt{d} (a^{-k/s } - a^{-k/t})  \to 0
\]
as $|s-t| \to 0$.  This proves \eqref{ensureclosed}.
\end{proof}

Since $\Delta$ is $\mathcal{F}_\sigma$ we may write it as $\Delta= \bigcup_n \Delta_n$ where each $\Delta_n$ is closed.  Let
	\[
	\Omega_n = \{ K(s,n) : s \in \Delta_n   \} \subset \mathcal{K}([0,1]^d)
	\]
and note that each $\Omega_n$ is closed by \eqref{ensureclosed} and the fact that each  $\Delta_n$ is closed.  Since $\mathcal{K}([0,1]^d)$ is separable in the Hausdorff metric, for each $n$ we can find a countable subset $\Omega_{n,0} \subseteq \Omega_n$ such that $\overline{\Omega_{n,0}} = \Omega_n$.  Let $\Omega^0 = \{ Q_1, Q_2, Q_3, \dots\}$ be an enumeration of $\bigcup_n\Omega_{n,0}$.  For technical reasons we further assume that if $\sup \Delta\in \Delta$ is an isolated point then  the set $Q_j$ with $\hd Q_j = \sup \Delta $ is repeated infinitely often in the sequence $(Q_i)_i$ (thus making $\Omega_0$ a multiset).

\begin{lma} \label{closed!}
We have
\[
\overline{\Omega^0} = \{Q_\infty\} \cup \bigcup_n \Omega_n.
\]
\end{lma}

\begin{proof}
  First note that $\overline{\Omega^0} \supseteq \bigcup_n \overline{\Omega_{n,0}} = \bigcup_n \Omega_n$  and $ Q_\infty \in \overline{\Omega^0} $ by construction (recall \eqref{1-over-n}) and so one direction is obvious.  The other direction is more difficult, but follows from the way we defined the sets $K(s,n)$.  It suffices to argue that $\{Q_\infty\} \cup \bigcup_n \Omega_n$ is a closed set.   Let $K(s_i, n_i) \in  \bigcup_n \Omega_n$ be a convergent sequence of sets, where for each $i$, $K(s_i, n_i) \in \Omega_{n_i}$.  By taking a subsequence if necessary  we may assume that either $n_i \to \infty$, in which case $K(s_i, n_i) \to Q_\infty$ by \eqref{1-over-n}, or $n_i=n$ is constant.  In this second case, we may take a further subsequence (using compactness of $\Delta_n$) where  $s_i \to s \in \Delta_n$.  Therefore  $K(s_i, n_i) \to K(s,n) \in \Omega_n$  by \eqref{ensureclosed}.  This proves the claim.
\end{proof}

We are now ready to build $F$, but there are two slightly different cases depending on whether or not $\inf \Delta = 0$. The basic  idea is to arrange shrinking copies of the sets $Q_i$ in such a way that a given miniset only sees a significant proportion of one of the sets $Q_i$, thus making the microsets easier to understand.   Using Lemma \ref{closed!} we then argue that microsets generated this way are essentially restricted to $\overline{\Omega^0}$ plus another set $Q_0$ which has dimension $\inf \Delta$. 

First, suppose that $\inf \Delta > 0$ and let $Q_0 \subseteq [0,1]^d$ be a self-similar set generated by homotheties and satisfying the strong separation condition which has Hausdorff dimension $\inf \Delta$. We now construct $F$ based on the structure of $Q_0$, see Figure \ref{fig:construction-of-F} for an illustration.  Suppose  $Q_0$ is generated by $b\geq 2$ similarity maps $\left\{h_{u}\right\}_{u=1}^b$ with common  contraction ratio $c_0$ and let $\mathcal{I}=\{1,\ldots, b\}$. Let $(\alpha_i)_{i\in \mathbb{N}}$ be a sequence of distinct integers which increases super exponentially, i.e. $\lim_{n\rightarrow \infty} \frac{\log \alpha_n}{n} = \infty $ and assume $\alpha_0=0$. Let 
	\begin{align*}
	\mathcal{I}_{\alpha_i}=\{(u_1,\ldots,u_{\alpha_i})\in \mathcal{I}^{\alpha_i} : \;&(u_1,\ldots,u_{\alpha_{i-1}}) = (1,\ldots,1) \text{ and} \\ &(u_{\alpha_{i-1}+1},\ldots,u_{\alpha_i}) \ne (1,\ldots,1) \}.
	\end{align*}
	Then for all $i\in \mathbb{N}$ define 
	\[
	Q^*_i= \bigcup_{(u_1,\ldots,u_{\alpha_i})\in \mathcal{I}_{\alpha_i}}h_{u_1}\circ h_{u_2}\circ \cdots \circ h_{u_{\alpha_i}}(Q_i)
	\]
	and let
	\[
	F = \overline{\bigcup_{i \in \mathbb{N}} Q_{i}^*}.
	\]
	
\begin{center}
\begin{figure}[!t]
  \begin{tikzpicture}[scale=0.3]

  \draw (0,9) -- (27,9);
  \draw (0,6) -- (9,6); \draw (18,6) -- (27,6);
  \draw (0,3) -- (3,3); \draw (6,3) -- (9,3);
  \draw (0,0) -- (1,0); \draw (2,0) -- (3,0); \draw (6,0) -- (7,0); \draw (8,0) -- (9,0); 

  \node[above] at (22.5,6) {$Q_1$};
  \node[above] at (2.5,0) {$Q_2$};
  \node[above] at (6.5,0) {$Q_2$};
  \node[above] at (8.5,0) {$Q_2$};

  \end{tikzpicture}
  \caption{Construction of $F$ with $Q_0$ being the middle third Cantor set, $\alpha_1=1$, and $\alpha_2=3$. Here, for example,  $Q_2^*$ consists of 3 copies of $Q_2$. }
  \label{fig:construction-of-F}
\end{figure}
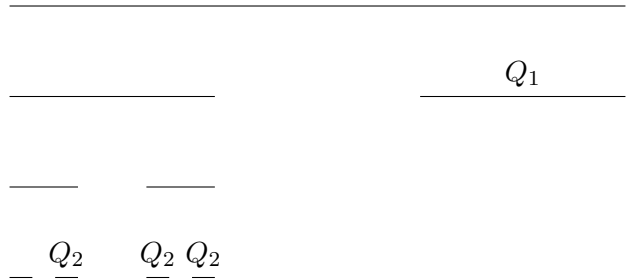
\end{center}

%
%
%
%
%
%
%

	We claim that $F$ has the desired properties. Since each $Q_i$ is a miniset of $F$, we clearly have
	\[
	\overline{\Omega^0} \subseteq \mathcal{G}_F.
	\]
Furthermore, since $d_{\mathcal{H}}(h_1^{-1} \circ \cdots \circ h_1^{-1}(F) \cap [0,1]^d, Q_0) \le c_0^{\alpha_{i+1}-\alpha_i}$ where $h_1^{-1}$ is composed $\alpha_i$ times, we see that $Q_0 \in \mathcal{G}_F$. Therefore, by Lemma \ref{closed!},
\[
  \Delta = \{ \hd E : E \in \{Q_0\} \cup \overline{\Omega^0} \} \subseteq \{ \hd E : E \in \mathcal{G}_F \}.
\]	
It remains to show that we do not get any `unwanted' microsets appearing whose Hausdorff dimension is outside of $\Delta$. Let us first deal with the (easy) case of microsets in $\mathcal{G}_F $ which are  actually  minisets, that is, they do not have   unbounded scaling sequence.  It suffices to consider the Hausdorff dimension of the intersection of $F$ with an arbitrary closed  cube.  Any cube whose interior intersects $F$ will either intersect only finitely many of the $Q_i^*$ in which case it will have the maximal Hausdorff dimension among this finite collection, or it will intersect $Q_i^*$ for all sufficiently large $i$ in which case it will have Hausdorff dimension $\sup \Delta$.  In either case the Hausdorff dimension belongs to $\Delta$.  We now consider  microsets which are not minisets.   Let $E \in \mathcal{G}_F$ and assume $E$ has unbounded scaling sequence.  Therefore we can find cubes $J_k$ and vectors 
\[
t_k=\left(\min_{(x_1,\ldots,x_d)\in J_k}x_1,\ldots,\min_{(x_1,\ldots,x_d)\in J_k}x_d\right)
\]
such that 
\[
\frac{\sqrt{d}}{\text{diam }J_k}(F \cap J_k - t_k) \to E
\]
and $\text{diam }J_k \to 0$.  If for all large enough $k$ we have that $J_k$ intersects more than one of the sets $Q_{i}^*$, then for  all but at most one of the sets $Q_{i}^*$, the constituent pieces $h_{u_1}\circ h_{u_2}\circ \cdots \circ h_{u_{\alpha_i}}(Q_i)$ become arbitrarily small compared with the diameter of $J_k$. Thus, for such $Q_i^*$, the portion intersecting $J_k$ will become arbitrarily close to a subset of $Q_0$ and so the microset obtained is just a subset of $Q_0$.  Therefore we may assume that $J_k$ intersects only one of the sets $Q_i^*$ for all large enough $k$, since the contribution from other sets either approaches a subset of $Q_0$, a singleton, or disappears completely.

Either the set of $i$ such that $Q_i^*$ intersects $J_k$ for some $k$ is bounded, in which case $E$ is a microset of one of the sets $Q_i$, and therefore $\hd E \in \Delta$, or the set of $i$ such that $Q_i^*$ intersects $J_k$ for some $k$ is unbounded, in which case $E$ is either a microset of $Q_0$ or a microset of a set from $\overline{\Omega^0}$, depending on how large the constituent pieces of $Q_i^*$ are with respect to each $J_k$. This uses the assumption that $\text{diam }J_k \to 0$.  Lemma \ref{closed!} implies that in all cases $\hd E \in \Delta$.  This completes the proof in the case $\inf \Delta>0$.

The proof in the case $\inf \Delta=0$ is similar, but actually more straightforward, and so we only sketch the idea.  We let
\[
Q_i^* = 2^{-i^i} Q_i +(2^{-i}, 0, \dots, 0)
\]
and
\[
F =\overline{ \bigcup_{i \in \mathbb{N}} Q_i^*} ,
\]
that is we scale the sets $Q_i$ by a superexponential factor and then  arrange them in an exponentially decreasing sequence accumulating at 0.  Arguing as above, a microset of $F$ with unbounded scaling sequence  is either  a microset of  the set $\{0\} \cup \{2^{-i} : i \in \mathbb{N}\}$  (which plays the role of $Q_0$ above) or  a microset of a set from $\overline{\Omega^0}$.

	\section{Small microsets}
	In this section, we will prove Theorem \ref{Sing}. 
	If $F$ has a genuine microset of zero Hausdorff dimension then it is clear that $F$ has zero lower dimension. Therefore we just need to show the other direction, this will be done by proving the contrapositive. Let $F$ be a compact subset of $\mathbb{R}^d$ such that the gallery of $F$ contains only microsets of cardinality at least two. To prove our result we just need to show $\Lower F >0$. 
	
	Let $k> 1$ be an integer. In what follows cubes are assumed to be oriented with the coordinate axes.  We say that $F$ satisfies property $P(k)$ if for every $x\in F$ and $R\in (0,1)$ the following statement is satisfied:
	\begin{itemize}
		\item[] If $Q(x,R)$ is the closed cube centred at $x$ with side length $R$, then there exist two cubes with disjoint interiors and with centres in $F \cap Q(x,R)$ and side lengths  $2^{-k}R$.
	\end{itemize}
	If $F$ fails property $P(k)$ for all  integers $k$ then for all $k$ there is a cube $Q_k$ such that $Q_k\cap F$ can be covered by one cube of side length $(2\sqrt{d})\, 2^{-k}$ times that of $Q_k$. It follows that  $T_k(F) \cap [0,1]^d$  converges in the Hausdorff metric to a singleton  as $i\to\infty$, where $T_k$ is the unique homothety mapping $Q_k$ to $[0,1]^d$. Therefore by our assumption we see that $F$ satisfies $P(k)$ for some integer  $k>1$, which we fix from now on.

	Let $x\in F$ be arbitrarily chosen and fix $0<r<R \leq 1$.  Consider $Q(x,R)$ and since $F$ satisfies $P(k)$ we see that there exist two disjoint cubes with centres in $F \cap Q(x,R)$ and side lengths  $2^{-k}R$.  Repeat the argument inside each of these cubes and then inside each of the four cubes at the next level and so on.  Run this argument $m$ times where $m$ is chosen to be the largest integer such that $2^{-km}R > r$.  It follows that there are
	\[
	2^m \geq 2^{-1} \left( \frac{R}{r} \right)^{1/k}
	\]
	disjoint cubes of side length at least $r$ contained in $Q(x,2R)$.  It follows that $\Lower F\geq 1/k>0$, as desired.

	\section{Further remarks and problems}
	
	\subsection{Dimensions which need not appear as dimensions of microsets} \label{dimappear}
	
	Here we elaborate on the following question: given a compact set $F \subset  \mathbb{R}^d$, which dimensions of $F$ necessarily appear as the Hausdorff dimension of a  microset of $F$ with unbounded scaling sequence?  The answer is: the lower and Assouad dimensions necessarily do, but the Hausdorff, packing and upper and lower box dimensions need not.   Note that it is vital to include the requirement that the scaling sequences are unbounded as otherwise the set itself appears as a microset and the question is trivial.  Concluding that the upper and lower box dimensions do not necessarily appear even in the closure of $\{ \hd E : E \in \mathcal{G}_F \text{ has unbounded scaling sequence}  \}$ is straightforward. For example, the set $\{1/n : n\in \mathbb{N} \}$ has  box dimension $1/2$, but all its microsets with unbounded scaling sequence are either an interval or a finite collection of equally spaced points.

	Concluding that the Hausdorff and packing  dimensions do not necessarily appear as Hausdorff dimensions of microsets with unbounded scaling sequence is a little more subtle and relies on our proof of Theorem \ref{Thmain2}.  Let $\Delta = \{1, 1/2-1/(n+1) : n = 1, 2, \dots\}$, which is clearly $\mathcal{F}_\sigma$,  and let $F\subset \mathbb{R}$ be the set constructed in the proof of Theorem \ref{Thmain2} given this $\Delta$.  Note that we may assume $\Omega_0 = \{Q_1,Q_2, \dots\}$ where $Q_n$ has dimension $1/2-1/(n+1)$ and importantly $Q_\infty \notin \Omega_0$.  It follows that $\hd F =\pd F = \sup_n \hd Q_n = 1/2 \notin \Delta$ as required.  The important point is that the $Q_n$ are chosen such that $\hd Q_n \to 1/2$, but $Q_n \to [0,1]$ in the Hausdorff metric. 
	
	Note that in the above example, the Hausdorff and packing dimensions do appear as accumulation points of the set of Hausdorff dimensions of microsets and so we pose the following question.
	
	\begin{ques}
		Is it true that if $F \subset  \mathbb{R}^d$ is compact, then $\hd F$ appears in the closure of $\{ \hd E : E \in \mathcal{G}_F \text{ has unbounded scaling sequence} \}$?
	\end{ques}

	\subsection{Hausdorff measures of microsets}
	Let us first recall that it is possible to obtain the following slightly stronger version of Theorem \ref{MicroAssouad}; see \cite[Theorem 1.3]{F2}.
	
	\begin{thm}
		Let $F$ be a compact set, then
		\[
		\Assouad F=\max \{s\geq 0 :  E\in\mathcal{G}_F \text{ and } \mathcal{H}^s(E)>0\}.
		\]
	\end{thm}
	
	Thus we can find a microset of $F$ whose $s$-Hausdorff measure is positive, where $s$ is the Assouad dimension of $F$. It is very natural to ask whether the following dual result for the lower dimension holds.
	
	\begin{ques}
		Is it true that if $F$ is compact, then
		\[
		\Lower F=\min \{s\geq 0 : E\in\mathcal{G}_F \text{ and } \mathcal{H}^s(E)<\infty\}?
		\]
	\end{ques}

	\subsection{Set theoretic complexity of $\{ \hd E : E \in \mathcal{G}_F \} $}
	
We proved that the set of dimensions attained by a gallery can be surprisingly complicated: it can be  $\mathcal{F}_\sigma$, despite the gallery itself being $\mathcal{F}$.  However, we are unaware if the set of attained dimensions can be any more complicated than $\mathcal{F}_\sigma$.  We remark that
\[
\Big\{ \{ \hd E : E \in \mathcal{G}_F \} \  : \  \text{$F \subseteq [0,1]^d$  compact }\Big\}
\]
must have cardinality $2^{\aleph_0}$, since we have a natural surjection from the set of  compact sets $F \subseteq [0,1]^d$ onto this set (this bounds the cardinality from above, and the lower bound is trivial). In particular, there must be sets  (which contain their infimum and supremum) which cannot be obtained as the set of dimensions attained by a gallery.

	\begin{ques}
		If  $F \subseteq [0,1]^d$ is compact, then is  $\{ \hd E : E \in \mathcal{G}_F \} \subseteq [0,d]$ an  $\mathcal{F}_\sigma$ set?  If not, does it belong to a finite Borel class?
	\end{ques}

	\providecommand{\bysame}{\leavevmode\hbox to3em{\hrulefill}\thinspace}
	\providecommand{\MR}{\relax\ifhmode\unskip\space\fi MR }
	\providecommand{\MRhref}[2]{%
		\href{http://www.ams.org/mathscinet-getitem?mr=#1}{#2}
	}
	\providecommand{\href}[2]{#2}

	\section*{Acknowledgments} 
	
	JMF was financially supported by a Leverhulme Trust Research Fellowship (RF-2016-500) and an EPSRC Standard Grant (EP/R015104/1).  DCH was financially supported by an EPSRC Doctoral Training Grant (EP/N509759/1). AK was financially supported by the Finnish Center of Excellence in Analysis and Dynamics Research, the Finnish Academy of Science and Letters, and the V\"ais\"al\"a Foundation. HY was financially supported by the University of St Andrews.  This research project began while all four authors were in attendance at the Institut Mittag-Leffler during the 2017  semester programme \emph{Fractal Geometry and Dynamics}.

  The authors thank Tamas Keleti and Yuval Peres for pointing out an error in an earlier version of this paper and Richard Balka and Michael Hochman for helpful discussions on the topic.

\end{document}